\documentclass[10pt]{amsart}

\usepackage{latexsym}
\usepackage{amsfonts}
\usepackage{amssymb}
\usepackage{amsmath}
\usepackage{mathrsfs}
\usepackage{hyperref}
\usepackage{tikz}
\usetikzlibrary{matrix,arrows}
\tikzstyle{shaded}=[fill=red!10!blue!20!gray!30!white]
\tikzstyle{shaded line}=[double=red!10!blue!20!gray!30!white, double distance=1.5mm, draw=black]
\tikzstyle{unshaded}=[fill=white]
\tikzstyle{unshaded line}=[double=white, double distance=1.5mm, draw=black]
\tikzstyle{Tbox}=[circle, draw, thick, fill=white, opaque,]
\tikzstyle{empty box}=[circle, draw, thick, fill=white, opaque, inner sep=2mm]
\tikzstyle{background rectangle}= [fill=red!10!blue!20!gray!40!white,rounded corners=2mm] 
\tikzstyle{on}=[very thick, red!50!blue!50!black]
\tikzstyle{off}=[gray]

\tikzstyle{traces}=[scale=.2, inner sep=1mm]
\tikzstyle{quadratic}=[scale=.4, inner sep=1mm, baseline]
\tikzstyle{annular}=[scale=.7, inner sep=1mm, baseline]
\tikzstyle{make triple edge size}= [scale=.4, inner sep=1mm,baseline] 
\tikzstyle{icosahedron network}=[scale=.3, inner sep=1mm, baseline]
\tikzstyle{ATLsix}=[scale=.25, baseline]
\tikzstyle{TL12}=[scale=.15,baseline]
\tikzstyle{PAdefn}=[scale=.7,baseline]
\tikzstyle{TLEG}=[scale=.5,baseline]

\title[Equivalence of Two Approaches to Yang-Mills]{Equivalence of Two Approaches to Yang-Mills on Non-commutative Torus}
\author{Partha Sarathi Chakraborty}
\address{The Institute of Mathematical Sciences, CIT Campus, Taramani, Chennai
600113}
\email{parthac@imsc.res.in}
\author{Satyajit Guin}
\address{The Institute of Mathematical Sciences, CIT Campus, Taramani, Chennai
600113}
\email{gsatyajit@imsc.res.in}
\date{\today}

\newtheorem{definition}{Definition}[section]
\newtheorem{theorem}[definition]{Theorem}
\newtheorem{lemma}[definition]{Lemma}
\newtheorem{proposition}[definition]{Proposition}

\newtheorem{remark}[definition]{Remark}

\begin{document}
\begin{abstract}
There are two notions of Yang-Mills action functional in noncommutative geometry. We show that for noncommutative n-torus both these notions agree. We also prove a structure theorem on the 
Hermitian structure of a finitely generated projective modules over spectrally invariant subalgebras of $C^*$-algebras.  
\end{abstract}
\maketitle
%\sffamily

%\vglue 1cm
\section{Introduction}
There are two approaches to noncommutative geometry due to Alain Connes. In the first approach \cite{1} one begins with $({\mathcal A}, G, \alpha, \tau)$ a $C^*$-dynamical system along with an invariant trace. Moreover one also assumes that the dynamics is governed by a Lie group. In this setting Connes introduces the basic notions like Hermitian vector bundles, connections, curvature etc. and eventually along with
 Rieffel (\cite{3}) he introduces the notion of Yang-Mils action functional as a certain functional $YM(\nabla) $ defined on the space of compatible connections 
$C({\mathcal E})$ on a finitely generated projective ${\mathcal A}$ module $\mathcal E$ with a Hermitian structure. Critical points of this functional has been studied by Rieffel in (\cite{3.1}). Variations of this concept have been studied in \cite{6}. Later (\cite{2}) Connes gave a spectral formulation of noncommutative geometry. In this formulation a noncommutative geometric space is
 described by a certain triple called spectral triple. This formulation is more popular today. In this setting as well he introduced the concept of compatible connections $\tilde{C}({\mathcal E})$ and Yang-Mills action functional. There is a general recipe to produce a ``spectral triple$\textquotedblright$ from a $C^*$-dynamical system. Here we have put spectral triple with in quotation because the general recipe does not tell you that the resulting object is a true spectral triple but they are candidates and one has to verify the relevant conditions on a case by case basis. However for noncommutative torus, the prime test case in the subject it is easy to see that indeed one obtains a genuine spectral triple.  Then one encounters the natural question whether these two notions agree. Connes addressed this for noncommutative two torus. Proposition 13, in the last chapter of \cite{2} states that the notions of compatible connections is same in both the approaches and the concepts of Yang-Mills action functionals also agree up to a normalizing factor. In this paper we take up the case of higher dimensional noncommutative torus and show that even in these cases the notions of compatible connections are same in the sense that there is an affine isomorphism between the spaces $C({\mathcal E})$ and $\tilde{C}({\mathcal E})$ and Yang-Mills action functionals also agree up to a normalizing factor. Along the way we also prove a structural result on finitely generated projective modules with Hermitian structure over spectrally invariant subalgebras of $C^*$-algebras. The result is the following. If $\mathcal A$ is a spectrally invariant subalgebra of a $C^*$-algebra $A$, that is ${\mathcal A} \subseteq A$ is a $*$-subalgebra  closed under holomorphic function calculus and $\mathcal E$ is a finitely generated projective $\mathcal A$ module with a Hermitian structure then there is a self adjoint projection $p \in M_n({\mathcal A})$ such that ${\mathcal E} \cong p {\mathcal A}^n$ and $\mathcal E$ has the induced Hermitian structure. %This fact is available in \cite{2} but we could not trace a correct proof in the literature and hence proved it here.
  Our proof makes crucial use of the hypothesis that ${\mathcal A}$ is closed under holomorphic function calculus. We do not have any counter example but we believe it is necessary to assume that $\mathcal A$ is spectrally invariant.

Organization of the paper is as follows. In section two we recall the definition of Yang-Mills action functional in the dynamical system approach and work out the expression for the noncommutative n-tori. Section three is devoted to the description of Yang-Mills in the frame work of spectral triples. We also show that finitely generated projective modules with a Hermitian structure over a spectrally invariant subalgebra of a $C^*$-algebra is always isomorphic with a submodule of a free module with induced Hermitian structure. Finally in the fourth section we explicitly work out the Yang-Mills functional for the noncommutative torus and show that it agrees with the first approach. This result is an higher dimensional analog of the corresponding result of Connes.
%\vglue 1cm
\section{First approach to Yang-Mills functional}

We  briefly recall the setting of (\cite{3}) for Yang-Mills functional on a $C^*$-dynamical system with an invariant, faithful trace. 
Let $(\mathcal{A},G,\alpha,\tau)$ be one such, where $\mathcal{A}$ is a unital $C^*$-alegebra, $G$ is a connected Lie group, $\alpha : G \longrightarrow Aut(\mathcal{A})$, a homomorphism such that for all $a \in {\mathcal A}$,
the mapping $g$ going to $\alpha_g(a)$ is continuous and $\tau$ is a $G$-invariant, faithful trace on $\mathcal A$. We say that $a \in \mathcal{A}$ is smooth iff the map $g \mapsto \alpha_g (a)$ from $G$ to the normed space
$\mathcal{A}$ is smooth. The involutive algebra $\mathcal{A}^{\infty} = \{ a \in \mathcal{A} :\, a \ \hbox{is smooth}  \}$ is a norm dense subalgebra of $\mathcal{A}$, called the  smooth subalgebra.
Note that this is unital as well. One crucial property enjoyed by this algebra is that it is closed under the holomorphic function calculus inherited from the ambient $C^*$-algebra $\mathcal A$.

Let $\mathcal{E}$ be a finitely generated projective module over $\mathcal{A}\,$. Unless otherwise stated we will only consider right modules. We will say f.g.p module to mean  finitely generated projective module.
There exists a f.g.p $\mathcal{A}^{\infty}$-module $\mathcal{E}^{\infty}$, unique upto isomorphism, such that $\mathcal{E}$ is isomorphic to $\mathcal{E}^{\infty} \otimes _{\mathcal{A}^{\infty}} \mathcal{A}$.
Conversely if $\mathcal{E}^{\infty}$ is a f.g.p $\mathcal{A}^{\infty}$-module then  $\mathcal{E}^{\infty} \otimes_{\mathcal{A}^{\infty}} \mathcal{A}$  becomes a f.g.p module over $\mathcal{A}$. 
Since we shall never work with $\mathcal{A}$ and $\mathcal{E}$ but only with $\mathcal{A}^{\infty}$ and $\mathcal{E}^{\infty}$,  from now on, for notational simplicity, we denote the latter by $\mathcal{A}$ and $\mathcal{E}$.
Define $\mathcal{E}^*$ as the space of $\mathcal{A}$ linear mappings from $\mathcal E$ to $\mathcal A$. Clearly $\mathcal{E}^*$ is a right $\mathcal A$ module.
\begin{definition}\label{def}
A {\it Hermitian} structure on $\mathcal{E}$ is an $\mathcal{A}$-valued positive-definite sesquilinear mapping 
$\langle \, \, , \, \rangle_{\mathcal{A}} $ such that, 
\begin{enumerate}
\item [(a)] $\langle \xi , \xi' \rangle_{\mathcal{A}}^* = \langle \xi' , \xi \rangle_{\mathcal{A}}\, , \, \, \, \forall \, \xi , \xi' \in \mathcal{E}$.
\item [(b)] $\langle \xi , \xi'. a \rangle_{\mathcal{A}} =  (\langle \xi , \xi' \rangle_{\mathcal{A}}) .a\, , \, \, \, \forall \, \xi , \xi' \in \mathcal{E},\, \, \forall \, a \in \mathcal{A}$.
\item [(c)] The map $\xi \longmapsto \Phi_\xi$ from $\mathcal{E}$ to $\mathcal{E}^*\,$, given by $\Phi_\xi(\eta) = \langle\xi,\eta\rangle_\mathcal{A}\, , \, \forall \eta \in \mathcal{E}\,$, gives an $\mathcal{A}$-module isomorphism
between $\mathcal{E}$ and $\mathcal{E}^*$. This property will be referred as the self-duality of $\mathcal E$.
\end{enumerate}
\end{definition}
Any free $\mathcal{A}$-module $\mathcal{E}_0=\mathcal{A}^q$ has a  {\it Hermitian} structure, given by 
$\langle\, \xi , \eta\,\rangle_\mathcal{A} = \sum_{j=1}^q \xi_j^* \eta_j , \forall\, \xi = (\xi_1, \ldots, \xi_q) \, , \, \eta = (\eta_1, \ldots, \eta_q) \in \mathcal{E}_0$. We refer this as the canonical  {\it Hermitian} structure
on $\mathcal{A}^q$. The following lemma shows that every f.g.p module admits a {\it Hermitian} structure.
\begin{lemma}\label{Herm-I}
(a) A f.g.p module of the form $p\mathcal{A}^q$, where $ p \in \mathcal{A} \otimes M_q(\mathbb {C})$ a projection, has a {\it Hermitian} structure.

(b) Every finitely generated projective module $\mathcal{E}$ over $\mathcal{A}$ is isomorphic as a f.g.p module with  $p\mathcal{A}^q$ where $p$ is a self-adjoint idempotent, that is a projection.
Hence $\mathcal{E}$ has a {\it Hermitian} structure on it.
\end{lemma}
\begin{proof}(a) With respect to the canonical Hermitian structure $\langle \, p^*\xi,\eta\, \rangle_\mathcal{A} = \langle\, \xi,p\eta\, \rangle_\mathcal{A}$ holds for any $p \in M_q(\mathcal{A})$.
Suppose $\mathcal{E}=p\mathcal{A}^q$ be a f.g.p module with $p$ a projection in $M_q(\mathcal{A})$. 
The canonical structure $\langle \, \xi,\eta\, \rangle_{\mathcal{A}}=\sum \xi_j^*\eta_j$ on $\mathcal{A}^q$ will induce a pairing on $\mathcal{E}$. 
We have to show that $\xi \longmapsto \Phi_{\xi}$ gives an
$\mathcal{A}$-module isomorphism between $\mathcal{E}$ and $\mathcal{E}^*$. It is enough to check only the surjectivity of this map. In order to do so let's take an element $f \in \mathcal{E}^*$.
Then $\tilde f = f \circ \pi$ is an element of $(\mathcal{A}^q)^*$ where $\pi: \mathcal{A}^q \longrightarrow p\mathcal{A}^q$ is the map $\xi \longmapsto p\xi$. By definition (\ref{def}) there exists $\eta
\in \mathcal{A}^q$ s.t $\tilde f=\Phi_\eta$. Consider any element $p\xi \in \mathcal{E}$ with $\xi \in \mathcal{A}^q$. Then,
\begin{eqnarray*}
f(p\xi) = f \circ \pi(p\xi) = \tilde f(p\xi) = \langle \, \eta,p\xi\, \rangle_\mathcal{A}
                                             = \langle \, p^*\eta,p\xi\, \rangle_\mathcal{A}
                                             = \langle \, p\eta,p\xi\, \rangle_\mathcal{A}
                                             = \Phi_{p\eta}(p\xi).
\end{eqnarray*}
Hence $ f= \Phi_{p\eta}$ with $\eta \in \mathcal{A}^q$. So the induced pairing on $\mathcal{E}$ gives a {\it Hermitian} structure on it.

(b) Let $\mathcal{E}$ be a f.g.p module over $\mathcal{A}\,$. There exists an $\mathcal{A}$-module $\mathcal{F}$ such that $\mathcal{E}\bigoplus \mathcal{F} \cong \mathcal{A}^q$ for some natural number $q$. 
Once we fix such an $\mathcal{F}$ we let $p : \mathcal{A}^q \longrightarrow \mathcal{A}^q$ given by $p(e+f) = e$ for $e\in \mathcal{E}$ and $f\in \mathcal{F}$. So $p$ is an idempotent in $M_q(\mathcal{A})$ 
with  $\mathcal{E} = p\mathcal{A}^q$. By (\cite{10},page $101$) we see that in a $C^*$-algebra $($or $*$-subalgebra of a $C^*$-algebra which is stable under holomorphic function calculus$)$ every idempotent 
is similar to a selfadjoint idempotent i.e., a projection and this similarity is witnessed by the invertible element $z={((2p^*-1)(2p-1)+1)}^{1/2}$. Since $\mathcal A$ is closed under holomorphic function calculus 
the invertible element $z$ actually belongs to $M_q(\mathcal{A})$. Hence $\tilde{p}  = z pz^{-1}$ is a projection in $M_q(\mathcal{A})$ and $\widetilde{\mathcal{E}} = \tilde {p}\mathcal{A}^q \cong  
p\mathcal{A}^q=\mathcal{E}$. Then one restricts the {\it Hermitian} structure on $\mathcal{A}^q$ to $\widetilde{\mathcal{E}}$ and endows $\mathcal{E}$ with the {\it Hermitian} structure obtained 
via the isomorphism between $\mathcal{E}$ and $\widetilde{\mathcal{E}}$. 
\end{proof}
\begin{remark} \rm
The concept of {\it Hermitian} structure can be defined for f.g.p modules over involutive algebras and part (a) of Lemma~(\,\ref{Herm-I}\,) still holds. But part (b) requires the more finer property of 
closure under holomorphic function calculus.
\end{remark}
\begin{remark}[{\bf  Open Question}] \rm We do not know whether a finitely generated projective module over an involutive algebra always admits Hermitian structure.
\end{remark}
Let $Lie(G)$ be the Lie algebra of $G$. Then we have a representation $\delta$ of $Lie(G)$ into the Lie algebra $Der(\mathcal{A})$ of derivations on $\mathcal{A}$ given by
\begin{eqnarray}
\delta_X(a)=\frac{d}{dt}{|}_{t=0}\, \alpha_{exp(tX)}(a).
\end{eqnarray}
where $exp: Lie(G) \rightarrow G$ is the exponential map.
\begin{definition}
Let $\mathcal{E}$ be a f.g.p module over $\mathcal{A}$ with a Hermitian Structure.
A connection (on $\mathcal{E}$) is a $\mathbb{C}$-linear map 
$\nabla : \mathcal{E} \longrightarrow \mathcal{E} \otimes Lie(G)^*$ such that,
for all $X \in Lie(G)$ and $\xi \in \mathcal{E}$, $a \in 
\mathcal{A}$ one has
\begin{eqnarray}
\nabla_X (\xi \cdot a) = \nabla_X (\xi) \cdot a + \xi \cdot \delta_X (a)  .
\end{eqnarray}
\end{definition}
We shall say that $\nabla$ is compatible with respect to the {\it Hermitian} structure on $\mathcal{E}$ 
iff~:
\begin{eqnarray*}
\langle \nabla_X \, \xi\, , \xi' \rangle_{\mathcal{A}} \, + \, \langle \xi\, , \nabla_X \, \xi' 
\rangle_{\mathcal{A}} = \delta_X (\langle\, \xi , \xi'\, \rangle_{\mathcal{A}}) \, , \quad \forall \, \xi\, , 
\xi' \in \mathcal{E},\, \, \forall \, X \in Lie(G) \, .
\end{eqnarray*} 
As discussed in (\cite{1}) compatible connection always exists. We will denote the set of compatible connections on $\mathcal{E}$ by C($\mathcal{E}$).
The algebra $End(\mathcal{E})$ has a natural involution given by$\,$,
\begin{eqnarray*}
\langle\, T^*\xi,\eta\, \rangle_\mathcal{A} = \langle\, \xi,T\eta\, \rangle_\mathcal{A}\, \, \, \, \, \, \, \forall\, \xi,\eta \in \mathcal{E}\, , \,T \in End(\mathcal{E}).
\end{eqnarray*}
For any two compatible connections $\nabla , \nabla^\prime$ it can be easily checked that $\nabla_X - \nabla^\prime_X$ is a skew-adjoint element of $End(\mathcal{E})$ for each $X \in Lie(G)$. The  curvature  $\varTheta_\nabla $ of a 
connection $\nabla$ is the alternating bilinear $End(\mathcal{E})$-valued form on $Lie(G)$ defined by,
\begin{eqnarray*}
\varTheta_\nabla : \wedge^2(Lie(G)) \longrightarrow End(\mathcal{E})
\end{eqnarray*}
\begin{eqnarray*}
\varTheta_\nabla (X \wedge Y) = [\nabla_X , \nabla_Y] - \nabla_{[X,Y]}\, , \, \, \forall\, X,Y \in  Lie(G).
\end{eqnarray*}
This measures the extent to which $\nabla$ fails to be a Lie algebra homomorphism. A simple calculation will assure that $\varTheta_\nabla$ takes value in $End(\mathcal{E})$. Infact more can be said about the range of $\varTheta_\nabla$.
We define $End(\mathcal{E})_{skew} = \{T \in End(\mathcal{E}) : T^* = -T \}$, the subset of skew-adjoint elements of $End(\mathcal{E})$.
\begin{lemma}\label{skew-adjoint}
Range of $\varTheta_\nabla$ is contained in $End(\mathcal{E})_{skew}$.
\end{lemma}
\begin{proof}
We have to show that $\langle\,\varTheta_\nabla(X\wedge Y)(\xi),\eta\,\rangle_\mathcal{A} = - \langle\,\xi,\varTheta_\nabla(X\wedge Y)(\eta)\,\rangle_\mathcal{A}\, \, $ for all $\xi,\eta \in \mathcal{E}$.
\begin{eqnarray*}
\langle\,\varTheta_\nabla(X\wedge Y)(\xi),\eta\,\rangle_\mathcal{A} & = & \langle\,([\nabla_X,\nabla_Y]-\nabla_{[X,Y]})(\xi),\eta\,\rangle_\mathcal{A}\\
                                                                    & = & \langle\,\nabla_X(\nabla_Y(\xi))-\nabla_Y(\nabla_X(\xi)) - \nabla_{[X,Y]}(\xi),\eta\,\rangle_\mathcal{A}\\
                                                                    & = & \langle\,\nabla_X(\nabla_Y(\xi)),\eta\,\rangle_\mathcal{A} - \langle\,\nabla_Y(\nabla_X(\xi)),\eta\,\rangle_\mathcal{A} - \langle\,\nabla_{[X,Y]}(\xi)),\eta\,\rangle_\mathcal{A}\\
                                                                    & = & \delta_X(\langle\,\nabla_Y(\xi),\eta\,\rangle_\mathcal{A}) - \langle\,\nabla_Y(\xi),\nabla_X(\eta)\,\rangle_\mathcal{A} - \delta_Y(\langle\,\nabla_X(\xi),\eta\,\rangle_\mathcal{A})\\
                                                                    &   & +\, \langle\,\nabla_X(\xi),\nabla_Y(\eta)\,\rangle_\mathcal{A} - \delta_{[X,Y]}(\langle\xi,\eta\rangle_\mathcal{A}) + \langle\,\xi,\nabla_{[X,Y]}(\eta)\,\rangle_\mathcal{A}\\
                                                                    & = & \delta_X(\delta_Y(\langle\xi,\eta\rangle_\mathcal{A})-\langle\xi,\nabla_Y(\eta)\rangle_\mathcal{A}) - \langle\nabla_Y(\xi),\nabla_X(\eta)\rangle_\mathcal{A})\\
                                                                    &   & -\, \delta_Y(\delta_X(\langle\xi,\eta\rangle_\mathcal{A})-\langle\xi,\nabla_X(\eta)\rangle_\mathcal{A}) + \langle\nabla_X(\xi),\nabla_Y(\eta)\rangle_\mathcal{A})\\
                                                                    &   & -\, \delta_{[X,Y]}(\langle\xi,\eta\rangle_\mathcal{A}) + \langle\xi,\nabla_{[X,Y]}(\eta)\rangle_\mathcal{A}\\
                                                                    & = & [\delta_X,\delta_Y](\langle\xi,\eta\rangle_\mathcal{A}) - \delta_{[X,Y]}(\langle\xi,\eta\rangle_\mathcal{A}) + \langle\xi,\nabla_{[X,Y]}(\eta)\rangle_\mathcal{A}\\
                                                                    &   & +\, \langle\nabla_X(\xi),\nabla_Y(\eta)\rangle_\mathcal{A} + \delta_Y(\langle\xi,\nabla_X(\eta)\rangle_\mathcal{A})\\
                                                                    &   & -\, \langle\nabla_Y(\xi),\nabla_X(\eta)\rangle_\mathcal{A} - \delta_X(\langle\xi,\nabla_Y(\eta)\rangle_\mathcal{A})\\
                                                                    & = & \langle\xi,\nabla_{[X,Y]}(\eta)\rangle_\mathcal{A} - \langle\xi,\nabla_X\nabla_Y(\eta)\rangle_\mathcal{A} + \langle\xi,\nabla_Y\nabla_X(\eta)\rangle_\mathcal{A}\\
                                                                    & = & \langle\xi,\nabla_{[X,Y]}(\eta)\rangle_\mathcal{A} - \langle\xi,[\nabla_X,\nabla_Y](\eta)\rangle_\mathcal{A}\\
                                                                    & = & -\, \langle\xi,([\nabla_X,\nabla_Y]-\nabla_{[X,Y]})(\eta)\rangle_\mathcal{A}\\
                                                                    & = & -\, \langle\,\xi,\varTheta_\nabla(X\wedge Y)(\eta)\,\rangle_\mathcal{A}\, \, \, .
\end{eqnarray*}
\end{proof}
We fix an inner product on $Lie(G)$ and this will remain fixed throughout. We next choose an orthonormal basis $\{Z_1,\ldots,Z_n\}\,$ of $Lie(G)$. The bilinear form on the space of alternating $2$-forms
with values in $End(\mathcal{E})$ is given by, 
\begin{eqnarray*}
\{\Phi,\Psi\}_\mathcal{E} = \displaystyle\sum_{i<j} \, \Phi(Z_i\wedge Z_j)\Psi(Z_i\wedge Z_j).
\end{eqnarray*}
Recall that we have a $G$-invariant faithful trace $\tau$ on $\mathcal{A}$. We can extend it to a canonical faithful trace $\widetilde \tau$ on $End(\mathcal{E})$ with the help of the following lemma.
\begin{lemma}
If $\,\mathcal{E}\,$ is f.g.p $\mathcal{A}$-module with a Hermitian structure, then every element of $End(\mathcal{E})$ can be written as a linear combination of elements of the form
$\langle\, \xi , \eta \, \rangle_\mathcal{E}$ for $\, \xi , \eta \in \mathcal{E}$, where $\langle\, \xi , \eta \, \rangle_\mathcal{E} (\zeta) = \xi \langle\, \eta , \zeta\, \rangle_\mathcal{A}, \, \,
\forall \, \zeta \in \mathcal{E}$.
\end{lemma}
\begin{proof}
Let $\,\mathcal{E} = p\mathcal{A}^q$ where $p \in M_q(\mathcal{A})$ is an idempotent and $\{e_1,\ldots,e_q\}$ be the standard basis for $\mathcal{A}^q$.
For any given $T \in End(\mathcal{E})$ one can write $\,T = \displaystyle\bigoplus_{i=1}^q T_i\, ,$ where $\,T_i = \pi_i \circ T,\, \pi_i$ denotes the projection onto the $i$-th component of
$\mathcal{A}^q$. Then $T_i(\xi) = \langle \eta_i,\xi \rangle_\mathcal{A}$ for some $\eta_i \in \mathcal{E}$, which follows from self duality of $\,\mathcal{E}$.
Then one can show directly that $T = \sum \langle\, pe_i,\eta_i\, \rangle_\mathcal{E}\, $.
\end{proof}
Now, using this lemma, we define a linear functional $\widetilde \tau$ on $End(\mathcal{E})$ as,
\begin{eqnarray*}
\widetilde{\tau} : End(\mathcal{E}) \longrightarrow \mathbb{C}
\end{eqnarray*}
\begin{eqnarray*}
\, \, \, \widetilde{\tau}(\langle\, \xi , \eta\, \rangle_\mathcal{E}) = \tau(\langle \eta , \xi \rangle_{\mathcal{A}}).\\
\end{eqnarray*}
\begin{lemma}
$\widetilde\tau$ defined above, is a trace on $End(\mathcal{E})$. 
\end{lemma}
\begin{proof}
One can easily check that $\langle\, \xi_1 , \eta_1\, \rangle_\mathcal{E} \langle\, \xi_2 , \eta_2\, \rangle_\mathcal{E} = \langle\,  \xi_1\langle\eta_1,\xi_2\rangle_\mathcal{A},\eta_2\, \rangle_\mathcal{E}$. 
Now use the fact that $\tau$ is a trace on $\mathcal{A}$.
\end{proof}
Moreover it can be shown that $\widetilde\tau$ is faithful (see \cite{3}). Finally, the Yang-Mills functional on $C(\mathcal{E})$ is given by,
\begin{eqnarray*}
\textit{YM}\, (\nabla) = -\widetilde{\tau} (\{ \varTheta_\nabla , \varTheta_\nabla \}_\mathcal{E})
\end{eqnarray*}
Notice that by Lemma~(\,\ref{skew-adjoint}\,) $\varTheta_\nabla$ takes value in $End(\mathcal{E})_{skew}$. Hence this minus sign will force $\textit{YM}\,$ to take nonnegative real values.
\medskip

Now we will deal with $\mathcal{A} = \mathcal{A}_\varTheta$, the noncommutative $n$-torus. We recall non-commutative $n$-torus $\mathcal{A}_\varTheta$ as defined in (\cite{8}).
Let $\varTheta$ be a $n \times n$ real skew-symmetric matrix. Denote by $\mathcal{A}_\varTheta$,  the universal $C^*$-algebra generated by $n\,$ unitaries $\, U_1,\ldots,U_n\,$ 
satisfying $U_k U_m = e^{2\pi i\varTheta_{km}} U_m U_k$, where $k,m \in \{1,\ldots,n\}$. Throughout this paper $i$ will stand for $\sqrt -1$. 
On the noncommutative $n$-torus $\mathcal{A}_\varTheta$, $G=\mathbb{T}^n$(connected Lie group) acts as follows:
\begin{eqnarray*}
\alpha_{(z_1,\ldots,z_n)}(U_k) =z_kU_k \, , \, k = 1,\ldots,n.
\end{eqnarray*}
The smooth subalgebra of $\mathcal{A}_\varTheta$, is given by
\begin{eqnarray*}
\mathcal{A}_\varTheta^\infty := \{\sum \, a_{\textbf{r}}\, U^{\textbf{r}} : \{ a_{\textbf{r}} \} \in \mathbb{S} (\mathbb{Z}^{n})\, , \, \textbf{r}=(r_1,\ldots,r_n) \in \mathbb{Z}^n \}
\end{eqnarray*}
where $\, \mathbb{S}(\mathbb{Z}^{n})$ denotes vector space of multisequences $(a_\textbf{r})$ that decay faster than the inverse of any polynomial 
in $\textbf{r} = (r_1,\ldots,r_n)$.

 This subalgebra is equipped with a $G$-invariant {\it tracial state}, given by $\tau(a) = a_\textbf{0} \, ,$ where $\textbf{0} = (0,\ldots,0)$. 
Then $\tau$ extends to a faithful trace on $\mathcal{A}_\varTheta$. We further assume that the lattice  $\Lambda_\varTheta$ generated by columns of $\varTheta$ is such that  $\,\Lambda_\varTheta + \mathbb{Z}^n\,$ 
is dense in $\mathbb{R}^n$. The advantage of choosing such a matrix $\,\varTheta\,$ is that, the tracial state $\tau$ becomes unique and $\mathcal{A}_\varTheta$ (hence $\mathcal{A}_\varTheta^\infty$) 
becomes simple (see \cite{5},\, Page $537$). The {\it Hilbert space} obtained by applying the G.N.S. construction to $\tau$ can be identified with $l^2(\mathbb{Z}^n)$.

From now on we will work with $\mathcal{A}_\varTheta^\infty$ only and hence for notational brevity we denote it by $\mathcal{A}_\varTheta$.
In this case $\mathcal{L} = Lie(G)\,$ is $\,\mathbb{R}^n$. Let $\{e_1,e_2,\ldots,e_n\}$ be the standard basis of $\,\mathbb{R}^n$ and the associated derivations $\delta_{e_1},\ldots,\delta_{e_n}$. We will denote $\delta_{e_j}$ by $\widetilde{\delta_j}$. 

The derivations $\{\widetilde{\delta_1},\ldots,\widetilde{\delta_n}\}$ on $\mathcal{A}_\varTheta$ are given by,
\begin{eqnarray}\label{induced actions}
\widetilde{\delta_j}(\displaystyle\sum_{\textbf{r}} \, a_{\textbf{r}} U^{\textbf{r}}) = i\displaystyle\sum_{\textbf{r}} \, r_j a_{\textbf{r}} U^{\textbf{r}}
\end{eqnarray}
It can be easily checked that these derivations commute and they are $^*$-derivations of $\mathcal{A}_\varTheta$ i,e.
\begin{eqnarray*}
(\widetilde\delta_j(a))^* = \widetilde\delta_j(a^*)\, \, \, \, ; \, \, \, \, \widetilde\delta_j(ab) = \widetilde\delta_j(a)b + a\widetilde\delta_j(b).
\end{eqnarray*}
A connection is given by $n$ maps $\nabla_{\widetilde{\delta_j}} : \mathcal{E} \longrightarrow \mathcal{E}$ such that $\nabla_{\widetilde{\delta_j}}(\xi.a) = \nabla_{\widetilde{\delta_j}}(\xi)a + \xi \widetilde{\delta_j}(a)$.
So the space of compatible connections $\nabla$ consists of $n$-tuples of maps $(\nabla_{\widetilde{\delta_1}},\ldots,\nabla_{\widetilde{\delta_n}})$ such that,
\begin{eqnarray}\label{connection-1}
\nabla(\xi) = \sum_{j=1}^n \nabla_{\widetilde{\delta_j}}(\xi) \otimes e_j. 
\end{eqnarray}
\begin{eqnarray}
\langle \nabla_{\widetilde{\delta_j}}(\xi),\eta \rangle_{\mathcal{A}_\varTheta} + \langle \xi, \nabla_{\widetilde{\delta_j}}(\eta) \rangle_{\mathcal{A}_\varTheta} & = & \widetilde{\delta_j}\left(\langle \xi,\eta \rangle_{\mathcal{A}_\varTheta}\right).
\end{eqnarray}

The {\it curvature} of a {\it connection} $\nabla$ is given by, $\varTheta_\nabla (\widetilde{\delta_j}\wedge\widetilde{\delta_k}) = [\nabla_{\widetilde{\delta_j}} , \nabla_{\widetilde{\delta_k}}]\,$ 
because $\,[\widetilde{\delta_j} , \widetilde{\delta_k}] = 0$ in this case. We have $[\nabla_{\widetilde{\delta_j}} , \nabla_{\widetilde{\delta_k}}]^* = - [\nabla_{\widetilde{\delta_j}} , \nabla_{\widetilde{\delta_k}}]$ 
by Lemma~(\,\ref{skew-adjoint}\,). The bilinear form on space of alternating $2$-forms with values in $End(\mathcal{E})$ becomes,
\begin{eqnarray*}
\{\Phi,\Psi\}_\mathcal{E} = \displaystyle\sum_{j<k}\Phi(\widetilde{\delta_j}\wedge\widetilde{\delta_k})
\Psi(\widetilde{\delta_j}\wedge\widetilde{\delta_k})
\end{eqnarray*}

Finally, the Yang-Mills functional of $\nabla$ is given by,
\begin{eqnarray*}
\textit{YM}\, (\nabla) = -\widetilde{\tau} (\{ \varTheta_\nabla , \varTheta_\nabla \}_\mathcal{E}) & = & -\widetilde{\tau} \left(\displaystyle\sum_{j<k}[\nabla_{\widetilde{\delta_j}},\nabla_{\widetilde{\delta_k}}]^2\right)\\ 
                                                                                                   & = & \widetilde{\tau} \left(\displaystyle\sum_{j<k}[\nabla_{\widetilde{\delta_j}},\nabla_{\widetilde{\delta_k}}]^*
[\nabla_{\widetilde{\delta_j}},\nabla_{\widetilde{\delta_k}}]\right)\\
                                                                                                   &  \geq 0 &.
\end{eqnarray*}
For notational simplicity we write,
\begin{eqnarray}\label{YM}
\textit{YM}(\nabla) = \displaystyle\sum_{j<k}\widetilde{\tau}\left([\nabla_j,\nabla_k]^*[\nabla_j,\nabla_k]\right)\,.
\end{eqnarray}
\medskip
 
\section{Second approach to Yang-Mills}

We first recall the differential graded algebra from \cite{2}.
\begin{definition}
A {\it spectral triple} $(\mathcal{A},\mathcal{H},D)$, over an algebra $\mathcal{A}$ with involution $*$, consists of the following things~:
\begin{enumerate}
\item a $\,*\,$-representation of $\mathcal{A}$ on a Hilbert space $\mathcal{H}$.
\item an unbounded selfadjoint operator $D$.
\item $D$ has compact resolvent and $[D,a]$ extends to a bounded operator on $\mathcal{H}$ for every $a \in \mathcal{A}$.
\end{enumerate}
\end{definition}
We shall assume that $\mathcal{A}$ is unital and the unit $1 \in \mathcal{A}$ acts as the identity on $\mathcal{H}$. If $|D|^{-d}$ is in the ideal of Dixmier traceable operators $\mathcal{L}^{(1,\infty)}$ then we say 
that the {\it spectral triple} is $(d,\infty)$-summable.

Let $\,\Omega^\bullet(\mathcal{A}) = \displaystyle\bigoplus_{k=0}^\infty \Omega^k(\mathcal{A})\,$ be the universal graded algebra over $\mathcal{A}$ where,
\begin{enumerate}
\item $\Omega^0(\mathcal{A}) = \mathcal{A}$
\item $\Omega^k(\mathcal{A}) = \{\, \displaystyle\sum_j a_0^jda_1^j\ldots da_k^j : a_r^j \in \mathcal{A}\, ; r=0,\ldots,k \}$.
\end{enumerate}
Here $d$ is an abstract operator such that,
\begin{enumerate}
\item $d1 = 0$
\item $d^2 = 0$
\item $d(\omega_1\omega_2) = (d\omega_1)\omega_2 + (-1)^{deg (\omega_1)} \omega_1 d\omega_2 \, , \, \, \forall \, \omega_j \in \Omega^\bullet(\mathcal{A})$.
\end{enumerate}
For $\omega \in \Omega^k(\mathcal{A})$ , $deg(\omega) = 0 \, , \, +1$ or $-1$ accordingly as whether $k$ is zero, even or odd. The involution $*$ of $\mathcal{A}$ extends uniquely to an involution on 
$\Omega^\bullet(\mathcal{A})$ by the rule $(da)^* = -da^*$ and $\pi$ also extends to a $*\,$-representation (again denoted by $\pi$) of $\,\Omega^\bullet(\mathcal{A})$ on $\mathcal{H}$ by,
\begin{eqnarray*}
\pi(a_0da_1\ldots da_k) = a_0[D,a_1]\ldots [D,a_k] \, ; \, a_j \in \mathcal{A}.
\end{eqnarray*}
Let $J_0^{(k)} = \{ \omega \in \Omega^k : \pi(\omega) = 0 \}$ and $J^\prime = \bigoplus J_0^{(k)}$. Since $J^\prime$ fails to be a differential graded ideal, the quotient $\, \Omega^\bullet/J^\prime$ is not a differential graded algebra. 
This problem can be overcome by letting $J^\bullet = \bigoplus J^{(k)}$ where $J^{(k)} = J_0^{(k)} + dJ_0^{(k-1)}$. Then $J^\bullet$ becomes a differential graded two-sided ideal and the quotient $\Omega_D^\bullet = \Omega^\bullet/J^\bullet$
becomes a differential graded algebra.

The representation $\pi$ gives an isomorphism,
\begin{eqnarray*}
\Omega_D^k \cong \pi(\Omega^k)/\pi(dJ_0^{k-1}).
\end{eqnarray*}
Arbitrary element of $\Omega_D^k$ can be viewed as a class of elements
\begin{eqnarray*}
\rho = \sum_j a_0^j \, [D, a_1^j ] \cdots [D, a_k^j]
\end{eqnarray*}
modulo the sub-bimodule of elements of the form,
\begin{eqnarray*}
\sum_j [D,b_0^j][D, b_1^j ] \cdots [D, b_{k-1}^j]  : b_r^j \in  \mathcal{A}\, ; \quad \sum_j b_0^j[D, b_1^j ] \cdots [D, b_{k-1}^j] = 0\, \, .
\end{eqnarray*}
Suppose we are given a unital $*\,$-algebra $\mathcal{A}$ and a $(d,\infty)$-summable spectral triple $(\mathcal{A},\mathcal{H},D)$ over $\mathcal{A}$.
Henceforth we will assume $\mathcal{A}$ is a subalgebra stable under holomorphic functional calculus in a $C^*$-algebra. Let $\mathcal{E}$ be a f.g.p (right)module over $\mathcal{A}$ equipped with 
a {\it Hermitian} structure on it. There is a right $\mathcal{A}$-module $\mathcal{F}$ such that $\mathcal{E}\bigoplus \mathcal{F}\cong \mathcal{A}^q\,$ for some $q$. Since $\mathcal{A}^q$ has a topology, 
$\mathcal{E}$ inherits the topology from $\mathcal{A}^q$. Also $\mathcal{E}^*$ inherits topology from $\mathcal{A}^q$ 
because $\mathcal{E}^*\bigoplus \mathcal{F}^*\cong {(\mathcal{A}^q)}^*\cong \mathcal{A}^q$. As because we have topology now, we can expect the isomorphism between $\mathcal{E}$ and $\mathcal{E}^*$ to be topological, 
which turns out to be true by the following lemma.    
\begin{lemma}\label{continuity}
If two finitely generated projective $\mathcal{A}$-modules $\mathcal{E}_1$ and $\mathcal{E}_2$ are algebraically isomorphic then they are topologically also isomorphic.
\end{lemma}
\begin{proof}
Since both the modules are projective, we can find $\mathcal{F}_1$ and $\mathcal{F}_2$ such that, $\mathcal{E}_1 \bigoplus \mathcal{F}_1 \cong \mathcal{A}^k$ and $\mathcal{E}_2 \bigoplus \mathcal{F}_2 \cong \mathcal{A}^l$. 
Then, $\mathcal{E}_1 \bigoplus \mathcal{F}_1 \bigoplus \mathcal{A}^l \cong \mathcal{A}^{k+l}$ and $\mathcal{E}_2 \bigoplus \mathcal{F}_2 \oplus \mathcal{A}^k \cong \mathcal{A}^{k+l}$. Hence we can write $\mathcal{E}_1 = p_1 
\mathcal{A}^{k+l}$ and $\mathcal{E}_2 = p_2 \mathcal{A}^{k+l}$, where $p_1,p_2 \in M_{k+l}(\mathcal{A})$ are idempotents. Let $u_j : \mathcal{A}^{k+l} \longrightarrow \mathcal{E}_j$ denote the projection maps and $v_j : 
\mathcal{E}_j \longrightarrow \mathcal{A}^{k+l}$ denote the inclusion maps for $j=1,2$. If we denote the isomorphism between $\mathcal{E}_1$ and $\mathcal{E}_2$ by $\,\phi\,$ then considering $f = v_2 \circ \phi \circ u_1$ 
and $g = v_1 \circ \phi^{-1} \circ u_2$ in $Hom(\mathcal{A}^{k+l},\mathcal{A}^{k+l})$, it is easily seen that $f \circ g = p_2$ and $g \circ f = p_1$. If we choose
\[ \tilde p_1 = 
\begin{pmatrix}
p_1 & 0\\
0 & 0
\end{pmatrix}
\, \, , \quad \tilde p_2 = 
\begin{pmatrix}
p_2 & 0\\
0 & 0
\end{pmatrix}
\]
\[ U = 
\begin{pmatrix}
f & 1-f\circ g\\
1-g\circ f & g
\end{pmatrix}
\]
then we see that $\tilde p_1 = U\tilde p_2 U^{-1}$.
U is a bounded operator acting on $\mathcal{A}\otimes \mathbb{C}^q$ where $q = 2(k+l)$. That is to say that the isomorphism is topological.
\end{proof}
\begin{lemma}
All Hermitian structures on a free module over $\mathcal{A}$ are isomorphic to each other.
\end{lemma}
\begin{proof}
The canonical {\it Hermitian} structure on $\mathcal{A}^q$ was given by $\langle \xi,\eta \rangle_{\mathcal{A}} = \sum_{k=1}^q \xi_k^*\eta_k$. We show that any other {\it Hermitian} structure 
is isomorphic to this one. Let $\langle\, \, , \, \rangle^\prime : \mathcal{A}^q \times \mathcal{A}^q \longrightarrow \mathcal{A}$ be another {\it Hermitian} structure on $\mathcal{A}^q$. Let $\{e_1,\ldots,e_q\}$ 
be standard basis of $\mathcal{A}^q$. Let $T=((t_{rs}))$ be given by $t_{sr}=\langle e_r,e_s\rangle^\prime$. Then $\langle \xi,\eta \rangle^\prime = \displaystyle\sum_{r,s} \langle e_r\xi_r,e_s\eta_s\rangle^\prime 
= \displaystyle\sum_{r,s} \xi_r^*\langle e_r,e_s\rangle^\prime\eta_s$. That is, $\langle \xi,\eta \rangle^\prime = \xi^*T\eta$, where $T \in M_q(\mathcal{A})$ is positive-definite. Hence $T$ is a positive element 
in the  $C^*$-algebra $M_q(\mathcal{B}(\mathcal{H}))$. Note that for $\xi \in \mathcal{A}^q$, $\xi^* = (\xi_1^*,\ldots,\xi_q^*)$ where $\xi = (\xi_1,\ldots,\xi_q)$. We consider elements of $\mathcal{A}^q$ 
as column vector, whereas their $*$ will denote row vector. So here $\xi^*$ is a row vector and $\xi$ is a column vector. We denote $\langle\, \, , \, \rangle^\prime$ by $\langle\, \, , \, \rangle_T$. 
Hence, {\it Hermitian} structures on $\mathcal{A}^q$ are parametrized by such $T$. We show that $T$ is one to one. Suppose $T\xi =0$. Then for any $\eta \in \mathcal{E}$, we get $\Phi_\xi(\eta) = \xi^*T\eta 
= (T\xi)^*\eta = 0$, showing $\Phi_\xi = 0$. Since $\xi \longmapsto \Phi_\xi$ is an isomorphism, we get $\xi =0$. Hence $T$ is one to one. To see $T$ is onto, we pick any $\zeta$ from $\mathcal{A}^q$. 
Then $\eta \longmapsto \zeta^*\eta$ is a $\mathcal{A}$-linear map on $\mathcal{A}^q$ taking value in $\mathcal{A}$ (we are dealing with right $\mathcal{A}$-module). Hence there exists $\xi$ in $\mathcal{A}^q$ such that, 
\begin{eqnarray*}
\Phi_\xi(\eta) = \zeta^*\eta = \xi^*T\eta = (T\xi)^*\eta
\end{eqnarray*}
Hence $\zeta = T\xi$, showing $T$ is onto. We define 
\begin{eqnarray*}
\widetilde T : \mathcal{A}^q \longrightarrow \mathcal{A}^q
\end{eqnarray*}
\begin{eqnarray*}
\widetilde T (T\xi) = \xi
\end{eqnarray*}
To show this map is continuous, let $T\xi_k \rightarrow T\xi$ in $\mathcal{A}^q$. Then $\xi_k^*T\eta \rightarrow \xi^*T\eta$ for any $\eta$ because multiplication is continuous with respect to the topology of 
$\mathcal{A}^q$. Hence $\Phi_{\xi_k} \rightarrow \Phi_{\xi}\, $. By Lemma~(\ref{continuity}\,)  $\xi \longmapsto \Phi_\xi$ is a continuous isomorphism. Hence we get $\xi_k \rightarrow \xi$, which shows 
continuity of $\widetilde T$. Thus $T$ has a bounded inverse $\widetilde T$ implying spectrum of $T$ is away from zero. Since $T$ is positive, $\sqrt T$ is a holomorphic function of $T$. Now define,
\begin{eqnarray*}
\Psi : \mathcal{A}^q \longrightarrow \mathcal{A}^q
\end{eqnarray*}
\begin{eqnarray*}
\Psi(\xi) = \sqrt{T}\xi
\end{eqnarray*}
Then, $\langle \xi , \eta\rangle_T = \xi^*T\eta = \xi^*\sqrt{T}\sqrt{T}\eta = \langle \Psi(\xi),\Psi(\eta)\rangle$. Since $\mathcal{A}$ is stable under holomorphic functional calculus in $\mathcal{B}(\mathcal{H})$, 
inverse of $\,T\,$ i,e. $\,\widetilde T\,$ lies in $M_q(\mathcal{A}) \subseteq M_q(\mathcal{B}(\mathcal{H}))\, $(see \cite{9}). Invertibility of $T$ in $M_q(\mathcal{A})$ gives invertibility of $\Psi$. So $\Psi$ 
gives an isomorphism between the canonical {\it Hermitian} structure $\langle\, \, ,\, \rangle_{\mathcal{A}}$ on $\mathcal{A}$ and {\it Hermitian} structure obtained throught $T$. Hence we are done.
\end{proof}
Using this lemma we can conclude the following fact about {\it Hermitian} structures on a f.g.p module which is also important in our calculation of Yang-Mills.
\begin{theorem}
Let $\mathcal{E}$ be a f.g.p $\mathcal{A}$-module with a {\it Hermitian} structure. Then we can have a self-adjoint idempotent $\, p \in M_q(\mathcal{A})$ such that $\,\mathcal{E} = p\mathcal{A}^q\,$ and $\,\mathcal{E}\,$ 
has the induced {\it Hermitian} structure.
\end{theorem}
\begin{proof}
Let $\mathcal{E}$ be a f.g.p $\mathcal{A}$-module with a {\it  Hermitian} structure $\langle\, \, ,\, \rangle_\mathcal{E}$. Because $\mathcal{E}$ is projective, we can have an $\mathcal{A}$-module $\mathcal{F}$ such that 
$\mathcal{E} \bigoplus \mathcal{F} \cong \mathcal{A}^q$ for some natural number $q$. Since $\mathcal{F}$ is also f.g.p $\, \mathcal{A}$-module, by Lemma~(\,\ref{Herm-I}\,) $\mathcal{F}$ has a {\it Hermitian} 
structure say $\langle\, \, ,\, \rangle_\mathcal{F}$. Then $\mathcal{E} \bigoplus \mathcal{F}$ posseses a Hermitian structure $\langle\, \, , \, \rangle$ given by,
\begin{eqnarray*}
\langle (e_1,f_1),(e_2,f_2)\rangle = \langle e_1,e_2\rangle_\mathcal{E} + \langle f_1,f_2\rangle_\mathcal{F}.
\end{eqnarray*}
i,e. we get a {\it Hermitian} structure on $\mathcal{A}^q$ coming from $\mathcal{E}$ and $\mathcal{F}$. By our previous lemma, this {\it Hermitian} structure is isomorphic with the canonical one.  Note that 
$\mathcal{E}$ is orthogonal to $\mathcal{F}$ with respect to this {\it Hermitian} structure. Let $p$ be a projection from $\mathcal{A}^q$ onto $\mathcal{E}$, i,e. $p(e+f) = e$. Then $\mathcal{E} = p\mathcal{A}^q$. Now,
\begin{eqnarray*}
\langle p(e_1+f_1),(e_2+f_2)\rangle & = & \langle e_1,e_2\rangle_E + \langle e_1,f_2\rangle_F\\
                                    & = & \langle e_1,e_2\rangle_E\\
                                    & = & \langle e_1,p(e_2+f_2)\rangle_E\\
                                    & = & \langle e_1+f_1,p(e_2+f_2)\rangle.
\end{eqnarray*}
which shows that $p$ is self-adjoint. Once we have a self-adjoint $p$, we can now restrict the {\it Hermitian} structure on $\mathcal{A}^q\,$ to $\,\mathcal{E}\,$ (recall proof of part $(a)$ of Lemma \ref{Herm-I}\,) 
and hence $\mathcal{E}$ has the induced {\it Hermitian} structure.
\end{proof}
Thus we see that any f.g.p $\mathcal{A}$-module has an induced {\it Hermitian} structure on it from a finitely generated free module.
Now we have the following definition.
\begin{definition}
Let $\mathcal{E}$ be a  Hermitian, f.g.p module over $\mathcal{A}$. A compatible connection on $\,\mathcal{E}\,$ is a $\, \mathbb{C}$-linear mapping
$\,\nabla : \mathcal{E} \longrightarrow \mathcal{E} \, \otimes _\mathcal{A} \Omega _{D}^1\,$ such that,
\begin{enumerate}
\item[(a)] $\nabla (\xi a) = (\nabla\xi)a + \xi \otimes da, \, \, \, \, \, \forall\, \xi \in \mathcal{E} , a \in \mathcal{A}$;
\item[(b)] $\langle \, \xi , \nabla \eta \, \rangle - \langle \, \nabla \xi , \eta \, \rangle = d\langle \, \xi , \eta \, \rangle_\mathcal{A}\, \, \, \, \, \, \, \forall\, \xi , \eta \in \mathcal{E}\, \, \, \,$(Compatibility).
\end{enumerate}
\end{definition}
The meaning of the last equality in $\Omega_D^1$ is, if $ \nabla(\xi ) = \sum\xi_j\otimes \omega_j $, with $\xi_j \in \mathcal{E}\, , \, \omega_j \in \Omega_D^1(\mathcal{A})$,
then $ \langle \nabla \xi, \eta\rangle = \sum \omega_j^* \langle\xi_j , \eta\rangle_\mathcal{A}$. Any f.g.p right module has a connection. An example of a compatible connection is the Grassmannian connection $\nabla_0$ on $\mathcal{E} 
= p\mathcal{A}^q$, given by $\nabla_0(\xi) = p d\xi$, where $d\xi= (d \xi_1, \ldots ,d\xi_q)$. This connection is compatible with the {\it Hermitian} structure,
\begin{eqnarray*}
\langle \xi,\eta \rangle_{\mathcal{A}} = \displaystyle\sum_{k=1}^q \xi_k^*\eta_k \, , \, \, \, \forall \, \xi,\eta \in p\mathcal{A}^q.
\end{eqnarray*}
Also, any two compatible connections can only differ by an element of $Hom_\mathcal{A}(\mathcal{E}\, ,\, \mathcal{E} \otimes_\mathcal{A} \Omega_D^1(\mathcal{A}))$. That is, the space of all compatible connections on 
$\,\mathcal{E}$, which we denote by $\widetilde C(\mathcal{E})$, is an affine space with associated vector
space $Hom_\mathcal{A}(\mathcal{E}\, ,\, \mathcal{E} \otimes_\mathcal{A} \Omega_D^1(\mathcal{A}))$.
The connection $\nabla$ extends to a unique linear map $\widetilde \nabla$ from
$\mathcal{E} \otimes \Omega_D^1$ to $\mathcal{E} \otimes \Omega_D^2$ such that,
\begin{eqnarray*}
\widetilde \nabla (\xi \otimes \omega) = (\nabla \xi)\omega + \xi \otimes \tilde d\omega, \, \, \, \forall \, \xi \in \mathcal{E}, \, \, \omega \in \Omega_D^1.
\end{eqnarray*}
It can be easily checked that $\widetilde\nabla$, defined above, satisfies the Leibniz rule, i,e.
\begin{eqnarray*}
\widetilde \nabla(\eta a) = \widetilde \nabla(\eta)a - \eta \tilde da \, , \, \, \forall \, a \in \mathcal{A},\eta \in \mathcal{E} \otimes \Omega_D^1\, . 
\end{eqnarray*}
A simple calculation shows that $\varTheta = \widetilde\nabla \circ \nabla$ is an element of $Hom_\mathcal{A}(\mathcal{E},\mathcal{E} \otimes_\mathcal{A} \Omega_D^2)$. Our next goal is to define an inner-product on 
$Hom_\mathcal{A}(\mathcal{E},\mathcal{E} \otimes_\mathcal{A} \Omega_D^2)$. In order to do so, recall that $\Omega_D^2 \cong \pi(\Omega^2)/\pi(dJ_0^{(1)})$. Let $\mathcal{H}_2$ be the Hilbert space completion of 
$\pi(\Omega^2)$ with the inner-product 
\begin{eqnarray*}
\langle T_1,T_2\rangle = Tr_\omega(T_1^*T_2|D|^{-d}),\, \forall \, T_1,T_2 \in \pi(\Omega^2). 
\end{eqnarray*}
Let $\widetilde{\mathcal{H}}_2$ be the Hilbert space completion of $\pi(dJ_0^{(1)})$ with the above inner-product. Clearly $\widetilde{\mathcal{H}}_2 \subseteq \mathcal{H}_2$. Let $P$ be the orthogonal projection of 
$\mathcal{H}_2$ onto the orthogonal complement of the subspace  $\pi(dJ_o^{(1)})$. Now define $\langle\, [T_1],[T_2]\,\rangle_{\Omega_D^2} = \langle PT_1,PT_2\rangle,\,$ for all $\, [T_j] \in \Omega_D^2$. 
This gives a well defined inner-product on $\Omega_D^2$. Viewing $\mathcal{E} = p\mathcal{A}^q$ we see that $Hom_\mathcal{A}(\mathcal{E},\mathcal{E} \otimes_\mathcal{A} \Omega_D^2) 
= Hom_\mathcal{A}(p\mathcal{A}^q,p\mathcal{A}^q \otimes_\mathcal{A} \Omega_D^2) \cong Hom_\mathcal{A}(pA^q,p(\Omega_D^2)^q)\,$, which is contained in $Hom_\mathcal{A}(\mathcal{A}^q,(\Omega_D^2)^q)$. 
Now for $\phi,\psi \in Hom_\mathcal{A}(\mathcal{E},\mathcal{E} \otimes_\mathcal{A} \Omega_D^2)$, define $\langle\langle \phi,\psi\rangle\rangle = \sum_k \langle \phi(p\widetilde{e_k}),\psi(p\widetilde{e_k})\rangle_{\Omega_D^2}$ 
where $\{\widetilde{e_1},\ldots,\widetilde{e_q}\}$ is the standard basis of $\mathcal{A}^q$. Finally, the Yang-Mills functional on $\widetilde C(\mathcal{E})$ is given by,
\begin{eqnarray}\label{main form}
\textit{YM}\,(\nabla) = \langle\langle\, \varTheta,\varTheta\, \rangle\rangle\, .
\end{eqnarray}

\section{Comparison between the two approaches}
In this section we work out the Yang-Mills action functional in the second formulation and show that this is same as the one coming from the first formulation.
Given a $C^*$-dynamical system $(\mathcal{A},G,\alpha)$ where $G$ is a connected Lie group with a $G$-invariant faithful trace $\tau$ on $\mathcal{A}$, we can consider the G.N.S Hilbert space 
$\widetilde{\mathcal{H}} = L^2(\mathcal{A},\tau)$. If {\it dimension} of the Lie group $G$ is $\,m$, letting $t = 2^{[m/2]}\,$, we know that there exist $m$ matrices in $M_t(\mathbb{C})$ denoted by 
$\gamma_1,\gamma_2,\cdots, \gamma_m\,$(called Clifford gamma matrices), such that, $\, \gamma_r\gamma_s + \gamma_s\gamma_r = 2\delta_{rs}\, ,\, r,s \in \{1,\ldots,m\}\, $, where $\delta_{rs}$ is the 
Kronecker delta function. In our case of non-commutative $n$ torus $\mathcal{A}_\varTheta$, the Lie group is $\, \mathbb{T}^n\,$ and hence we get $n$ Clifford gamma matrices $\gamma_1,\ldots ,\gamma_n\,$. 
We define $D := \sum_{j=1}^n \delta_j \otimes \gamma_j$ where $\delta_j = (-i)\widetilde\delta_j\, \,$(recall definition of $\widetilde\delta_j$ 
from \ref{induced actions}). Then $D$ becomes self-adjoint on $\mathcal{H} = \widetilde{\mathcal{H}} \otimes \mathbb{C}^N$ with domain $\mathcal{A}_\varTheta \otimes \mathbb{C}^N$, $N=2^{[n/2]}$. Moreover $|D|^{-n}$ lies in $\mathcal{L}^{(1,\infty)}$ with $Tr_\omega(|D|^{-n}) = 2N\pi^{n/2}/(n(2\pi)^n\Gamma(n/2))\,$(see \cite{5},Page $545$) 
and  $(\mathcal{A}_\varTheta,\mathcal{H},D)$ gives us a $(n,\infty)$-summable spectral triple. Following propositions determine the $\mathcal{A}_\varTheta$-bimodules $\Omega_D^1$ and $\Omega_D^2$ upto bimodule isomorphisms~:
\begin{proposition}
$\Omega_D^1 \cong \underbrace {\mathcal{A}_\varTheta \oplus \ldots \oplus \mathcal{A}_\varTheta}_{n\, \,times}$.
\end{proposition}
\begin{proof}
We know that $\Omega_D^1 \cong \pi(\Omega^1)$. Let $\omega \in \Omega^1$, so
$\omega = \sum_j \, a_jdb_j, \, a_j,b_j \in \mathcal{A}_\varTheta$. Then,
\begin{eqnarray*}
\pi(\omega) & = & \displaystyle\sum_j (a_j\otimes I)[D,b_j]\\
            & = & \displaystyle\sum_j \left(\displaystyle\sum_{l=1}^n a_j\delta_l(b_j)\otimes \gamma_l\right).
\end{eqnarray*}
Since $\{\gamma_1,\ldots,\gamma_n\} \subseteq M_N(\mathbb{C})$ is a linearly independent set, their linear span forms a $n$-dimensional vector space $\mathbb{C}^n$ where we identify $\gamma_l$ 
with $\alpha_l=(0,\ldots,1,\ldots,0) \in \mathbb{C}^n$ with $1$ in the $l$-th place. $\{\alpha_1,\ldots,\alpha_n\}$ is the canonical basis for $\mathbb{C}^n$. Hence we get that $\Omega_D^1 \subseteq 
\mathcal{A}_\varTheta\otimes\mathbb{C}^n$. For any $a \in \mathcal{A}_\varTheta$, we can write $a = aU_l^*U_l = aU_l^*\delta_l(U_l)$.  Now equality follows from simpleness of $\mathcal{A}_\varTheta$ together with the
fact that $\pi(\Omega^1)$ is $\mathcal{A}_\varTheta$-bimodule. 
\end{proof}
\begin{remark} \rm
Henceforth throughout this article $\{\sigma_1,\ldots,\sigma_n\}$ will denote the standard basis of $\mathcal{A}_{\varTheta}^n$ as free $\mathcal{A}_{\varTheta}$-bimodule where $\sigma_k 
= \underbrace {(0,\ldots,1,\ldots,0)}_{n\, tuple}$ with $1$ in the $k$-th place; whereas $\{\widetilde{e}_1,\ldots,\widetilde{e}_q\}$ will stand for the standard basis of $\mathcal{A}_{\varTheta}^q$ 
where $\widetilde{e}_l = \underbrace {(0,\ldots,1,\ldots,0)}_{q\, tuple}$ with $1$ in the $l$-th place. We will reserve this notation in the rest of this article. Under the identification in the 
above proposition, $\sigma_k$ is identified with $U_k^*\delta_k(U_k) \otimes \gamma_k$ in $\Omega_D^1$ for $k \in \{1,\ldots,n\}$.
\end{remark}
\begin{proposition}
$\Omega_D^2 \cong \underbrace{\mathcal{A}_\varTheta \oplus \ldots \ldots \oplus \mathcal{A}_\varTheta}_{n(n-1)/2\, \, times}$.
%$\Omega_D^2 \cong \displaystyle\bigoplus_{l=1}^{n(n-1)/2} \mathcal{A}_\varTheta^{(l)},\, $ where $\mathcal{A}_\varTheta^{(l)} = \mathcal{A}_\varTheta , \, \, \forall \, l$. 
\end{proposition}
\begin{proof}
We know that $\Omega_D^2 \cong \pi(\Omega^2)/\pi(dJ_0^{(1)})$. Let $\omega \in \Omega^2$ and write $\omega = \sum_r a_rdb_rdc_r\,$,$\,$ where $a_r,b_r,c_r \in \mathcal{A}_\varTheta\,$. Then,
\begin{eqnarray*}
\pi(\omega) & =  & \displaystyle\sum_r (a_r\otimes I)[D,b_r][D,c_r]\\
            & =  & \displaystyle\sum_r \left(a_r\otimes I\right)            
                        \left(\displaystyle\sum_{j=1}^n \delta_j(b_r)\otimes \gamma_j\right) 
                        \left(\displaystyle\sum_{k=1}^n \delta_k(c_r)\otimes \gamma_k\right)\\
            & =  & \displaystyle\sum_r \left(\displaystyle\sum_{j=1}^n 
                        a_r\delta_j(b_r)\otimes \gamma_j\right)\left(\displaystyle\sum_{k=1}^n  
                        \delta_k(c_r)\otimes \gamma_k\right)\\
            & =  & \displaystyle\sum_r \left(\left(\displaystyle\sum_{j=1}^n  
                        a_r\delta_j(b_r)\delta_j(c_r)\otimes I\right) + \displaystyle\sum_{p<q} \left(a_r\delta_p(b_r)\delta_q(c_r) - a_r\delta_q(b_r)\delta_p(c_r)\right)\otimes \gamma_p \gamma_q\right)
\end{eqnarray*}

Since we know that, $\gamma_l^2 = I$ and $\gamma_l \gamma_m = - \gamma_m \gamma_l$ for $l \neq m$, $a_r\delta_p(b_r)\delta_q(c_r)\otimes \gamma_p \gamma_q + a_r\delta_q(b_r)\delta_p(c_r)\otimes 
\gamma_q \gamma_p = (a_r\delta_p(b_r)\delta_q(c_r) - a_r\delta_q(b_r)\delta_p(c_r))\otimes \gamma_p \gamma_q$. Now $\gamma_l\gamma_m$ is independent with all $\gamma_p\gamma_q$ if $l,m \notin \{p,q\}$. 
Hence, $\pi(\Omega^2) \subseteq \displaystyle\bigoplus_{l=1}^{1+n(n-1)/2} \mathcal{A}_\varTheta^{(l)} ,\, $ where $\mathcal{A}_\varTheta^{(l)} = \mathcal{A}_\varTheta , \, \, \forall \, l$ because total 
number of the elements $(a_r\delta_p(b_r)\delta_q(c_r)\otimes \gamma_p \gamma_q - a_r\delta_q(b_r)\delta_p(c_r)\otimes \gamma_p \gamma_q)$ is $n(n-1)/2$. To show equality we use simpleness of $\mathcal{A}_\varTheta$. 
For that purpose we take any non-zero $a \in \mathcal{A}_\varTheta$ and take $b=U_1\, ,\, c=U_1^*$. Then $adU_1 dU_1^* \in \Omega^2$ and $\pi(adU_1 dU_1^*) 
= -a\otimes I$ is a non-zero element of $\pi(\Omega^2)$. Similarly for each $p,q$ we consider $U_q^*U_p^*d(U_p)d(U_q) \in \Omega^2$. Then $\pi(U_q^*U_p^*d(U_p)d(U_q)) = 1\otimes \gamma_p \gamma_q$. Hence for each
$p,q$ it is possible to choose nontrivial element of $\Omega^2$ s.t. the coefficient $\displaystyle\sum_r \left(a_r\delta_p(b_r)\delta_q(c_r) - a_r\delta_q(b_r)\delta_p(c_r)\right)$ of $\gamma_p\gamma_q$ is nonzero.
Now using the fact that $\pi(\Omega^2)$ is $\mathcal{A}_\varTheta$ bimodule and $\mathcal{A}_\varTheta$ is simple we get eqality.

Now we calculate $\pi(dJ_0^{(1)})$. We have $\omega \in J_0^{(1)}$ implies $\omega =\sum_s a_sdb_s$ where $a_s,b_s \in \mathcal{A}_\varTheta$, such that $\displaystyle\sum_s (a_s\otimes I)[D,b_s] = 0$. 
So we get, $\displaystyle\sum_s (a_s\otimes I)\left(\displaystyle\sum_{j=1}^n \delta_j(b_s)\otimes\gamma_j\right) = 0$, that is, $\displaystyle\sum_{j=1}^n \left(\displaystyle\sum_s a_s\delta_j(b_s)\right)\otimes \gamma_j = 0$. 
But,$\, \gamma_1,\ldots,\gamma_n$ being linearly independent we get,
\begin{eqnarray}\label{calc}
\displaystyle\sum_s a_s\delta_j(b_s)\otimes\gamma_j = 0,\, \, \, \, \forall j=1,\ldots,n.
\end{eqnarray}
Now $d\omega = \displaystyle\sum_s da_sdb_s$. So,
\medskip
\begin{eqnarray*}
\pi(d\omega) & = & \displaystyle\sum_s [D,a_s][D,b_s]\\
             & = & \displaystyle\sum_s \left(\displaystyle\sum_{j=1}^n \delta_j(a_s)\otimes\gamma_j\right)\left(\displaystyle\sum_{k=1}^n \delta_k(b_s)\otimes\gamma_k\right)\\
             & = & \displaystyle\sum_s \left((\displaystyle\sum_{j=1}^n \delta_j(a_s)\delta_j(b_s)\otimes I) +\ldots + (\delta_p(a_s)\delta_q(b_s) - \delta_q(a_s)\delta_p(b_s))\otimes \gamma_p\gamma_q +\ldots\right)
\end{eqnarray*}
Now from eqn. (\,\ref{calc}\,) we get,
\begin{eqnarray*}
\displaystyle\sum_s\delta_p(a_s)\delta_q(b_s)\otimes\gamma_p\gamma_q = -\displaystyle\sum_s a_s\delta_p\delta_q(b_s)\otimes\gamma_p\gamma_q
\end{eqnarray*}
and,
\begin{eqnarray*}
\displaystyle\sum_s\delta_q(a_s)\delta_p(b_s)\otimes\gamma_q\gamma_p = -\displaystyle\sum_s a_s\delta_q\delta_p(b_s)\otimes\gamma_q\gamma_p
\end{eqnarray*}
Hence,
\begin{eqnarray*}
(\delta_p(a_s)\delta_q(b_s) - \delta_q(a_s)\delta_p(b_s))\otimes \gamma_p\gamma_q  & = & (-a_s\delta_p\delta_q(b_s)+a_s\delta_q\delta_p(b_s))\otimes \gamma_p\gamma_q\\
                  & = & 0
\end{eqnarray*}
because, $\, \delta_p\delta_q = \delta_q\delta_p\, , \, \, \forall p,q\in \{1,\ldots,n\}$. Hence, $\pi(dJ_0^{(1)}) \subseteq \mathcal{A}_\varTheta$ and to show the equality we produce a non-trivial element in 
$\pi(dJ_0^{(1)})$ and again use the simpleness of $\mathcal{A}_\varTheta$. Consider $\omega = a(U_1^*dU_1 - 1/2 \times U_1^{-2}d(U_1^2)) \in \Omega^1$ with $a \neq 0$. Then we get $\pi(\omega) =0$ but $\pi(d\omega) 
= a\otimes I \neq 0\, \,$(which also shows non-triviality of $\omega$). Hence we conclude $\pi(dJ_0^{(1)}) \cong \mathcal{A}_\varTheta$. 
\end{proof}
Now we want to determine the differential $\tilde{d} : \pi(\mathcal{A}_\varTheta) \longrightarrow \Omega_D^1\,$ so that, $\tilde{d}(\pi(a)) = \pi(da) , \, \, \forall a\in \mathcal{A}_\varTheta$. 
\begin{lemma}
$\tilde{d} : \pi(\mathcal{A}_\varTheta) \longrightarrow \Omega_D^1$ is given by, $\pi(a) \longmapsto (\delta_1(a),\ldots,\delta_n(a))$. 
\end{lemma}
\begin{proof}
Pick any element $\pi(a) \in \pi(\mathcal{A}_\varTheta)$. Then $da \in \Omega^1$ and hence $\pi(da) = [D,a] = \displaystyle\sum_{j=1}^n \delta_j(a)\otimes\gamma_j$. 
This is an element in $\Omega_D^1$, which is isomorphic to $\mathcal{A}_\varTheta^n$ and under this isomorphism, $\displaystyle\sum_{j=1}^n \delta_j(a)\otimes\gamma_j\, $ goes to $(\delta_1(a),\ldots,\delta_n(a))$ 
in $\mathcal{A}_\varTheta^n$. Hence the above definition of $\tilde{d}$ is justified.
\end{proof}
Next we want to determine the differential $\tilde{d} : \Omega_D^1 \longrightarrow \Omega_D^2\,$ so that, $\tilde{d}(\pi(\omega)) = \pi(d\omega),\, \, \forall \, \omega \in \Omega^1$.
\begin{lemma}
$\tilde{d} : \Omega_D^1 \longrightarrow \Omega_D^2\, $ is given by,
\begin{eqnarray*}
(0,\ldots,a,\ldots,0) \longmapsto ((\delta_p(aU_j^*)\delta_q(U_j)-(\delta_q(aU_j^*)\delta_p(U_j)))_{1\leq  p<q\leq n}
\end{eqnarray*}
for $a$ in the $j$-th place.
\end{lemma}
\begin{proof}
For $(0,\ldots,a,\ldots,0) \in \Omega_D^1\,$ with $a$ in the $j$-th place, we have $\, aU_j^*dU_j \in \Omega^1$, such that $\pi(aU_j^*dU_j)$ is identified with $(0,\ldots,a,\ldots,0)$. 
Now, $d(aU_j^*dU_j) = d(aU_j^*)dU_j\, $, an element of $\Omega^2$. Now,
\begin{eqnarray*}
\pi(d(aU_j^*)dU_j) & = & [D,aU_j^*][D,U_j]\\
                   & = & \left(\displaystyle\sum_{l=1}^n \delta_l(aU_j^*)\otimes \gamma_l\right)\left(\displaystyle\sum_{k=1}^n \delta_k(U_j)\otimes \gamma_k\right)\\
                   & = & \displaystyle\sum_{l=1}^n \delta_l(aU_j^*)\delta_l(U_j) \otimes I + \displaystyle\sum_{p<q}\left(\delta_p(aU_j^*)\delta_q(U_j) - \delta_q(aU_j^*)\delta_p(U_j)\right)\otimes \gamma_p\gamma_q.
\end{eqnarray*}
Under the isomorphism $\Omega_D^2 \cong \mathcal{A}_\varTheta^{n(n-1)/2}$, $\displaystyle\sum_{p<q}(\delta_p(aU_j^*)\delta_q(U_j) - \delta_q(aU_j^*)\delta_p(U_j))\otimes \gamma_p\gamma_q\, $ 
goes to the required point in $\mathcal{A}_\varTheta^{n(n-1)/2}\, \, $.
\end{proof}
Finally the product map is recognized by the following lemma.
\begin{lemma}\label{product map}
The product map $\, \widetilde\prod : \Omega_D^1 \times \Omega_D^1 \longrightarrow \Omega_D^2\,$ is given by,
\begin{center}
$(a_1,\ldots,a_n).(b_1,\ldots,b_n) := ((a_pb_q - a_qb_p))_{1\leq p<q\leq n}$
\end{center}
\end{lemma}
\begin{proof}
We have a product $\prod : \Omega^1 \times \Omega^1 \longrightarrow \Omega^2\,$ given by $\prod(a_1da_2,b_1db_2) = a_1da_2b_1db_2 = a_1d(a_2b_1)db_2 - a_1a_2db_1db_2\,$. Choose two elements 
$(a_1,\ldots,a_n)$ and $(b_1,\ldots,b_n)$ in $\Omega_D^1$. We have seen previously that $\pi(\displaystyle\sum_{m=1}^n a_mU_m^*d(U_m))$ in $\pi(\Omega^1)$ is identified with $(a_1,\ldots,a_n)$. 
Similarly for $b_m$ inplace of $a_m$. Let $\omega = \displaystyle\sum_{m=1}^n a_mU_m^*d(U_m)$ and $\omega^\prime = \displaystyle\sum_{m=1}^n b_mU_m^*d(U_m) $. Now,
\begin{eqnarray*}
\prod(\omega,\omega^\prime) & = & \left(\displaystyle\sum_{m=1}^n a_mU_m^*d(U_m)\right)\left(\displaystyle\sum_{j=1}^n b_jU_j^*d(U_j)\right)\\
                            & = & \displaystyle\sum_{m,j=1}^n a_mU_m^*d(U_m)b_jU_j^*d(U_j)\\
                            & = & \displaystyle\sum_{m,j=1}^n \left(a_mU_m^*d(U_mb_jU_j^*)d(U_j)- a_md(b_jU_j^*)d(U_j)\right)
\end{eqnarray*}
It is an element of $\Omega^2$. Applying $\pi$ on it we get
\begin{eqnarray*}
\pi(\prod(\omega,\omega^\prime)) & = & \displaystyle\sum_{m,j=1}^n \left(a_mU_m^* [D,U_mb_jU_j^*][D,U_j] - a_m[D,b_jU_j^*][D,U_j]\right)\\
                                 & = & \displaystyle\sum_{m,j=1}^n (a_mU_m^*(\displaystyle\sum_{k=1}^n \delta_k(U_mb_jU_j^*)\otimes \gamma_k)(\displaystyle\sum_{l=1}^n \delta_l(U_j)\otimes \gamma_l)\\
                                 &   & - a_m(\displaystyle\sum_{r=1}^n \delta_r(b_jU_j^*)\otimes \gamma_r)(\displaystyle\sum_{s=1}^n \delta_s(U_j)\otimes \gamma_s))\\
                                 & = & \displaystyle\sum_{p<q}(\displaystyle\sum_{m,j=1}^na_mU_m^*\delta_p(U_mb_jU_j^*)\delta_q(U_j)-\displaystyle\sum_{m,j=1}^na_mU_m^*\delta_q(U_mb_jU_j^*)\delta_p(U_j)\\
                                 &   & - \displaystyle\sum_{m,j=1}^na_m\delta_p(b_jU_j^*)\delta_q(U_j) + \displaystyle\sum_{m,j=1}^n a_m\delta_q(b_jU_j^*)\delta_p(U_j))\otimes \gamma_p\gamma_q\\                                
\end{eqnarray*}
For each $p$ and $q$,
\begin{eqnarray*}
\displaystyle\sum_{p<q} (\displaystyle\sum_{m,j=1}^n (a_mU_m^*\delta_p(U_mb_jU_j^*)\delta_q(U_j)-\displaystyle\sum_{m,j=1}^n a_mU_m^*\delta_q(U_mb_jU_j^*)\delta_p(U_j)\\
- \displaystyle\sum_{m,j=1}^n a_m\delta_p(b_jU_j^*)\delta_q(U_j) + \displaystyle\sum_{m,j=1}^n a_m\delta_q(b_jU_j^*)\delta_p(U_j))
\end{eqnarray*}
\begin{eqnarray*}
& = & \displaystyle\sum_{m=1}^n (a_mU_m^*\delta_p(U_mb_qU_q^*)U_q - a_m\delta_p(b_qU_q^*)U_q\\
&   & + a_m\delta_q(b_pU_p^*)U_p -a_mU_m^*\delta_q(U_mb_pU_p^*)U_p)\\
& = & \displaystyle\sum_{m=1}^n (a_mU_m^*\delta_p(U_mb_q) - a_m\delta_p(b_q)\\
&   & + a_m\delta_q(b_p) - a_mU_m^*\delta_q(U_mb_p))\\
& = & \displaystyle\sum_{m=1}^n (a_mU_m^*\delta_p(U_m)b_q - a_mU_m^*\delta_q(U_m)b_p)\\
& = & a_pb_q - a_qb_p.
\end{eqnarray*}
Hence for $(a_1,\dots,a_n),(b_1,\dots,b_n) \in \Omega_D^1$, we get $\widetilde\prod((a_1,\dots,a_n),(b_1,\dots,b_n)) = ((a_pb_q - a_qb_p))_{1\leq p<q\leq n}\,$.
\end{proof}
It can be easily checked that both the $\tilde d$ , defined above, are derivations. We first prove the following lemmas which will help us in the computation.
\begin{lemma}\label{D-trace}
The canonical trace $\tau$ on $\mathcal{A}_\varTheta$ equals $1/Tr_\omega(|D|^{-n}) \int $ where $Tr_\omega$ denotes Dixmier trace and $\int a := Tr_\omega((a\otimes I)|D|^{-n})\, $
for all $\, a \in \mathcal{A}_\varTheta$.
\end{lemma}
\begin{proof}
We have $\tau (a) = \tau (\alpha _{\textbf{g}} (a)), \, \, \forall \, \textbf{g} \in \mathbb{T}^n$ because $\tau$ is $G$-invariant on $\mathcal{A}_\varTheta$. The G.N.S $\, $ Hilbert space $L^2(\mathcal{A}_\varTheta,\tau)$ 
is identified with $l^2(\mathbb{Z}^n)$. For $\textbf{g} \in \mathbb{T}^n$, $\alpha_{\textbf{g}}(U_1^{k_1} \ldots U_n^{k_n}) = \textbf{g}^{\textbf{k}} U_1^{k_1} \ldots U_n^{k_n}$. 
Here $\textbf{g}=(g_1,\ldots,g_n) \in \mathbb{T}^n\, $; $\textbf{g}^{\textbf{k}} = g_1^{k_1} \ldots g_n^{k_n}$. Define,
\begin{eqnarray*}
U_\textbf{g} : \, L^2(\mathcal{A}_{\varTheta},\tau) \longrightarrow L^2(\mathcal{A}_{\varTheta},\tau)
\end{eqnarray*}
\begin{eqnarray*}
a \longmapsto \alpha_\textbf{g}(a)
\end{eqnarray*}
It is easy to check this map is isometry with dense range. Hence extends as unitary on $L^2(\mathcal{A}_\varTheta,\tau)$. For $e_\textbf{k} \in l^2(\mathbb{Z}^n)$, $U_\textbf{g}(e_\textbf{k}) 
= \textbf{g}^{\textbf{k}}e_{\textbf{k}}$. Since $D(e_{\textbf{k}} \otimes M) = \sum_{j=1}^n k_j e_{\textbf{k}} \otimes \gamma_j M$ for $M \in M_N(\mathbb{C})$, it follows that $D (U_{\textbf{g}}\otimes I) = (U_{\textbf{g}}\otimes I) D$ 
on $L^2(\mathcal{A}_\varTheta \, , \tau)\otimes \mathbb{C}^N$. But $(U_{\textbf{g}}\otimes I) D (U_{\textbf{g}}^*\otimes I) = D \Rightarrow (U_{\textbf{g}}\otimes I) |D| (U_{\textbf{g}}^*\otimes I) = |D|$ which further implies $(U_{\textbf{g}}\otimes I) |D|^{-n}(U_{\textbf{g}}^*\otimes I) = |D|^{-n}$. Hence, 
\begin{eqnarray*}
Tr_\omega ((U_{\textbf{g}}aU_{\textbf{g}}^*\otimes I) |D|^{-n}) = Tr_\omega ((U_{\textbf{g}}\otimes I)(a\otimes I)|D|^{-n}(U_{\textbf{g}}^*\otimes I)) = Tr_\omega ((a\otimes I)|D|^{-n})
\end{eqnarray*}
which shows that $1/Tr_\omega(|D|^{-n}) \int\, $ is also a $G$-invariant trace on $\mathcal{A}_\varTheta$. Now uniqueness of $G$-invariant trace on $\mathcal{A}_\varTheta$ gives 
$\tau(a) = Tr_\omega((a\otimes I)|D|^{-n})/Tr_\omega(|D|^{-n})$ where $Tr_\omega(|D|^{-n})$ is a positive constant.
\end{proof}
\begin{lemma}\label{spin}
If $\{\gamma_1,\ldots,\gamma_n\}$ are Clifford gamma matrices in $M_N(\mathbb{C})$ then they enjoys the property $Trace(\gamma_l\gamma_m) = 0$ for $l \neq m$.
\end{lemma}
\begin{proof}
This follows immediately from the fact that Clifford gamma matrices satisfy the relation $\gamma_l\gamma_m + \gamma_m\gamma_l = 2\delta_{lm}$ for all $l,m$.
\end{proof}
\begin{lemma}\label{trace}
The positive linear functional $\int : T \longmapsto Tr_\omega (T|D|^{-n})/Tr_\omega (|D|^{-n})$, for $T \in M_N(\mathcal{A}_\varTheta)$, equals with $\tau \otimes Trace\, $, where $`Trace$' denotes the ordinary matrix trace (normalized) on $M_N(\mathbb{C})$.
\end{lemma}
\begin{proof}
Since $D^2 = \sum \delta_j^2 \otimes I_N\, $, $|D|^{-n}$ commutes with $1 \otimes M_N({\mathbb C})$ it follows that $\int$ is a trace on $M_N(\mathcal{A}_\varTheta) \cong \mathcal{A}_\varTheta \otimes M_N(\mathbb{C})$.

Our requirement is now fulfilled because of the fact that $\tau \otimes Trace$ is the unique extention
(normalized) of $\tau$ on $M_N(\mathcal{A}_\varTheta)$.
\end{proof}
\begin{lemma}\label{ortho}
If $l \neq m$ then any $a \otimes \gamma_l\gamma_m$ lies in the range of $P$ where $P$ was the orthogonal projection onto the orthogonal complement of $\pi(dJ_o^{(1)}) \subseteq \Omega_D^2$.
\end{lemma}
\begin{proof}
Recall that any element of $\pi(dJ_0^{(1)})$ looks like $x \otimes I$. Now $\langle\, a \otimes \gamma_l\gamma_m\, ,\, x\otimes I\, \rangle_{\pi(\Omega^2)} = Tr_\omega ((a^*x \otimes \gamma_l\gamma_m)|D|^{-n}) 
= Tr_\omega(|D|^{-n})\tau(a^*x) Trace\, (\gamma_l\gamma_m) = 0$, since $Trace\, (\gamma_l\gamma_m) = 0$ by Lemma~(\,\ref{spin}\,).
\end{proof}

Now we are ready to calculate the Yang-Mills for $\mathcal{A}_\varTheta$. Since $\Omega_D^1 \cong \mathcal{A}_{\varTheta}^n$, any compatible connection $\nabla : \mathcal{E} \longrightarrow \mathcal{E}\otimes \Omega_D^1\, $ 
is given by $n$-tuple of maps $(\nabla_1,\dots,\nabla_n)$, where $\nabla_j : \mathcal{E} \longrightarrow \mathcal{E}\, $ such that,
\begin{eqnarray}\label{connection-2}
\nabla(\xi) = \displaystyle\sum_{j=1}^n \nabla_j(\xi)\otimes \sigma_j 
\end{eqnarray}
\begin{eqnarray}
\langle \xi,\nabla_j(\eta) \rangle - \langle \nabla_j(\xi),\eta \rangle & = & \delta_j(\langle \xi,\eta \rangle_{\mathcal{A}_\varTheta}).
\end{eqnarray}
Here $\{\sigma_1,\ldots,\sigma_n\}$ is the standard basis of $\mathcal{A}_\varTheta^n$ as free $\mathcal{A}_{\varTheta}$-bimodule. Then $\widetilde\nabla : \mathcal{E}\otimes \Omega_D^1 \longrightarrow \mathcal{E}\otimes\Omega_D^2\,$ 
is given by, $\, \widetilde\nabla(\xi\otimes \sigma_m) = \left(\displaystyle\sum_{j=1}^n \nabla_j(\xi)\otimes \sigma_j\right) \sigma_m + \xi\otimes\tilde d(\sigma_m)\,$ for each $m=1,\ldots,n$.
\begin{proposition}\label{curvature}
The {\it curvature} $\, \varTheta = \widetilde\nabla \circ \nabla\, $ is given by $\displaystyle\sum_{m< j} [\nabla_m,\nabla_j](.)\otimes \sigma_m\sigma_j$ where $\sigma_m,\sigma_j \in \mathcal{A}_{\varTheta}^n\,$ and  
$\sigma_m\sigma_j$ is the element in $\mathcal{A}_{\varTheta}^{n(n-1)/2}$ produced by the product map $\widetilde{\prod}$ of $\, Lemma~(\,\ref{product map})$.
\end{proposition}
\begin{proof}
Through direct computation we get,
\begin{eqnarray*}
\varTheta(\xi) & = & \widetilde\nabla \circ \nabla(\xi)\\
               & = & \displaystyle\sum_{m=1}^n \widetilde\nabla(\nabla_m(\xi)\otimes \sigma_m)\\
               & = & \displaystyle\sum_m \left((\displaystyle\sum_j \nabla_j(\nabla_m(\xi))\otimes \sigma_j)\sigma_m + \nabla_m(\xi)\otimes \tilde d(\sigma_m)\right)\\
               & = & \displaystyle\sum_{m,j} \nabla_j(\nabla_m(\xi))\otimes \sigma_j\sigma_m + \nabla_m(\xi)\otimes \tilde d(\sigma_m)\\
               & = & \displaystyle\sum_{m< j} [\nabla_m,\nabla_j](\xi)\otimes \sigma_m\sigma_j + \displaystyle\sum_m \nabla_m(\xi)\otimes \tilde d(\sigma_m).
\end{eqnarray*}
But, $\displaystyle\sum_m \nabla_m(\xi)\otimes \tilde d(\sigma_m) = \displaystyle\sum_m \nabla_m(\xi)\otimes ((\delta_p(U_m^*)\delta_q(U_m) - \delta_q(U_m^*)\delta_p(U_m))_{1\leq p<q\leq n} = 0$ because 
$\delta_j(U_m^*) = -U_m^*\delta_j(U_m)U_m^*$. Hence $\,\varTheta = \displaystyle\sum_{m<j}[\nabla_m,\nabla_j]\,\otimes \,\sigma_m\sigma_j\,$.
\end{proof}
\begin{proposition}\label{main}
$\textit{YM}(\nabla) = \displaystyle\sum_{m<j} \tau_q([\nabla_m,\nabla_j]^*[\nabla_m,\nabla_j])$ upto a positive factor where $\tau_q$ denotes the extended trace $\tau \otimes Trace$ on $M_q(\mathcal{A}_\varTheta)$.
\end{proposition}
\begin{proof}
Recall that $\langle\langle \varTheta,\varTheta\rangle\rangle = \displaystyle\sum_{k=1}^q\, \langle\, \varTheta(p\widetilde e_k),\varTheta(p\widetilde e_k)\, \rangle_{\Omega_D^2}$ where $\{\widetilde e_1,\ldots,\widetilde e_q\}$ 
denotes standard basis of $\mathcal{A}_{\varTheta}^q$ and $\mathcal{E} = p\mathcal{A}_{\varTheta}^q$. Let $[\nabla_m,\nabla_j](p\widetilde e_k) = \eta^{(mjk)} = p\eta^{(mjk)} = (\eta_1^{(mjk)},\ldots,\eta_q^{(mjk)}) 
\in \mathcal{A}_\varTheta^q$. Then from proposition (\,\ref{curvature}\,) we get 
\[\varTheta(p\widetilde e_k) = \sum_{m<j}(\eta_1^{(mjk)}\sigma_m\sigma_j,\ldots,\eta_q^{(mjk)}\sigma_m\sigma_j),\] an element of 
$(\Omega_D^2)^q$. It is easily seen that as $\mathbb{C}$ vector spaces $Hom(\mathcal{E},\mathcal{E} \otimes \Omega_D^2) \cong \bigoplus Hom(\mathcal{E},\mathcal{E})$. We can view $End(\mathcal{E})$ as 
$pM_q(\mathcal{A}_\varTheta)p \subseteq M_q(\mathcal{A}_\varTheta)$. We have an inner-product on $\bigoplus M_q(\mathcal{A}_\varTheta)$ given by $\langle (A_1,\ldots,A_t),(B_1,\ldots,B_t)\rangle 
= \displaystyle\sum_{j=1}^t \tau_q(A_j^*B_j)$. Following calculation shows this inner-product becomes same with the one on $Hom(\mathcal{E},\mathcal{E} \otimes \Omega_D^2)$.
\begin{eqnarray*}
\langle \varTheta(p\widetilde e_k),\varTheta(p\widetilde e_k)\rangle
                                                  & = & \displaystyle\sum_{m<j,l<r} \, \langle\, (\eta_1^{(mjk)} \sigma_m\sigma_j,\ldots,\eta_q^{(mjk)} \sigma_m\sigma_j)\, ,\, 
(\eta_1^{(lrk)} \sigma_l\sigma_r,\ldots,\eta_q^{(lrk)} \sigma_l\sigma_r)\, \rangle\\
                                                  & = & \displaystyle\sum_{m<j,l<r} \displaystyle\sum_{s=1}^q \, \langle\, \eta_s^{(mjk)} \sigma_m\sigma_j\, ,\, \eta_s^{(lrk)} \sigma_l\sigma_r\, 
\rangle_{\Omega_D^2}\\
                                                  & = & \displaystyle\sum_{m<j,l<r} \displaystyle\sum_{s=1}^q \, \langle\, [\eta_s^{(mjk)}\otimes \gamma_m\gamma_j]\, ,\, [\eta_s^{(lrk)}\otimes 
\gamma_l\gamma_r]\, \rangle_{\Omega_D^2}\\
                                                  & = & \displaystyle\sum_{m<j,l<r} \displaystyle\sum_{s=1}^q \, \langle \, P(\eta_s^{(mjk)}\otimes \gamma_m\gamma_j)\, ,\, P(\eta_s^{(lrk)}\otimes 
\gamma_l\gamma_r)\, \rangle_{\pi(\Omega^2)}\\
                                                  & = & \displaystyle\sum_{m<j,l<r} \displaystyle\sum_{s=1}^q Tr_\omega \left(({\eta_s^{(mjk)}}^* \eta_s^{(lrk)}\otimes \gamma_j\gamma_m
\gamma_l\gamma_r)|D|^{-n}\right).
\end{eqnarray*}
Last equality follows from Lemma (\ref{ortho}). Now use of Lemma (\ref{spin}) and (\ref{trace}) shows the following,
\begin{eqnarray*}
\langle \langle \varTheta,\varTheta \rangle\rangle   & = & Tr_\omega(|D|^{-n}) \displaystyle\sum_{k=1}^q \displaystyle\sum_{m<j} \displaystyle\sum_{s=1}^q\, \tau\left({\eta_s^{(mjk)}}^*\eta_s^{(mjk)}\right)\\
                                                     & = & Tr_\omega(|D|^{-n})\displaystyle\sum_{k=1}^q \displaystyle\sum_{m<j}\, \tau\left(\, \langle\, [\nabla_m,\nabla_j](p\widetilde e_k)\, ,\, [\nabla_m,\nabla_j]
(p\widetilde e_k)\, \rangle_{\mathcal{A}_\varTheta}\, \right)\\
                                                     & = & Tr_\omega(|D|^{-n})\displaystyle\sum_{k=1}^q \displaystyle\sum_{m<j}\, \tau\left(\, \langle\, p\widetilde e_k\, ,\, [\nabla_m,\nabla_j]^*[\nabla_m,\nabla_j]
(p\widetilde e_k)\, \rangle_{\mathcal{A}_\varTheta}\, \right)\\
                                                     & = & Tr_\omega(|D|^{-n})\displaystyle\sum_{m<j}\, \tau_q\, \left(\, [\nabla_m,\nabla_j]^*[\nabla_m,\nabla_j]\, \right).
\end{eqnarray*}
The last equality follows from the fact that for any $T=((t_{rs})) \in pM_q(\mathcal{A}_\varTheta)p\, $ where $\,p \in M_q(\mathcal{A}_\varTheta)$ is a projection, 
$\displaystyle\sum_{k=1}^q \, \langle \widetilde e_k,T\widetilde e_k\rangle_{\mathcal{A}_\varTheta} = \displaystyle\sum_{k=1}^q \, \langle\, p\widetilde e_k,Tp\widetilde e_k\, \rangle_{\mathcal{A}_\varTheta} 
= \displaystyle\sum_{r=1}^q t_{rr}\,$. Hence follows the proposition.
\end{proof}
Recall that $\{e_1,\ldots,e_n \}$ denotes the standard basis choosen for $\mathbb{R}^n$ and $\{\sigma_1,\ldots,\sigma_n\}$ is the standard basis of $\Omega_D^1$. We have an one to one correspondence between these sets, 
both being finite sets of same cardinality. The following theorem points out the main result. 
\begin{theorem}
Let $C(\mathcal{E})$ and $\widetilde{C}(\mathcal{E})$ denote the affine space of compatible connections for the first and second approaches respectively. Then both these are in one to one correspondence through
an affine isomorphism and the value of Yang-Mills functional on corresponding elements of these spaces are same upto a positive scalar factor. That is to say the following diagram commutes.
\begin{center}
\begin{tikzpicture}[node distance=3cm,auto]
\node (D){};
\node (A)[node distance=1.5cm,left of=D]{$C(\mathcal{E})$};
\node (B)[node distance=1.5cm,right of=D]{$\widetilde{C}(\mathcal{E})$};
\node (C)[node distance=2cm,below of=D]{$\mathbb{R}_+ \cup \{0\}$};
\draw[->](A) to node{{$\Phi$}}(B);
\draw[->](B)to node{{\tiny $\textit{YM}$}}(C);
\draw[->](A)to node[swap]{{\tiny $\textit{cYM}$}}(C);
\end{tikzpicture} 
\end{center}
where $c=2N\pi^{n/2}/(n(2\pi)^n\Gamma(n/2)).$
\end{theorem}
\begin{proof}
Recall from equation (\ref{connection-1}) for any $\nabla \in C(\mathcal{E})$, $\nabla(\xi) = \sum_{j=1}^n \nabla_j(\xi)\otimes e_j\,$ where $\nabla_j : \mathcal{E} \longrightarrow \mathcal{E}$. We define $\Phi(\nabla) = \widetilde{\nabla}\,$ where,
\begin{eqnarray*}
\widetilde{\nabla}(\xi) = \displaystyle\sum_{j=1}^n (-i)\nabla_j(\xi)\otimes \sigma_j
\end{eqnarray*}
It is easy to see that $\widetilde{\nabla}$ defines a connection. Given compatibility of $\nabla\,$, we have to check whether $\widetilde{\nabla}$ is compatible with respect to the {\it Hermitian} structure. 
This follows from a direct calculation.
\begin{eqnarray*}
\langle\, \xi,\widetilde{\nabla}(\eta)\, \rangle - \langle\, \widetilde{\nabla}(\xi),\eta\, \rangle & = & \displaystyle\sum_{j=1}^n\, (\langle\, \xi,(-i)\nabla_j(\eta)\otimes \sigma_j\, \rangle - \langle\, 
(-i)\nabla_j(\xi)\otimes \sigma_j,\eta\, \rangle\, )\\
                                       & = & \displaystyle\sum_{j=1}^n\, (\, \langle\, \xi,\nabla_j(\eta)\, \rangle_{\mathcal{A}_\varTheta}(-i)\sigma_j - i\sigma_j^*\langle\, \nabla_j(\xi),\eta\, 
\rangle_{\mathcal{A}_\varTheta})\\
                                       & = & \displaystyle\sum_{j=1}^n\, (\, \langle\, \xi,\nabla_j(\eta)\, \rangle_{\mathcal{A}_\varTheta}(-i)\sigma_j - i\sigma_j\langle\, \nabla_j(\xi),\eta\, 
\rangle_{\mathcal{A}_\varTheta})\\
                                       & = & (-i)(\langle\, \xi,\nabla_1(\eta)\rangle_{\mathcal{A}_\varTheta} + \langle\nabla_1(\xi),\eta\, \rangle_{\mathcal{A}_\varTheta},\ldots,\\
                                       &   &  \langle\, \xi,\nabla_n(\eta)\rangle_{\mathcal{A}_\varTheta} + \langle\nabla_n(\xi),\eta\, \rangle_{\mathcal{A}_\varTheta})\\
                                       & = & (-i)(\widetilde{\delta_1}(\, \langle\, \xi,\eta\, \rangle_{\mathcal{A}_\varTheta}),\ldots,\widetilde{\delta_n}(\, \langle\, \xi,\eta\, \rangle_{\mathcal{A}_\varTheta}))\\
                                       & = & (\delta_1(\, \langle\, \xi,\eta\, \rangle_{\mathcal{A}_\varTheta}),\ldots,\delta_n(\, \langle\, \xi,\eta\, \rangle_{\mathcal{A}_\varTheta}))\\
                                       & = & \tilde d\,(\, \langle\, \xi,\eta\, \rangle_{\mathcal{A}_\varTheta}).
\end{eqnarray*}
which shows compatibility of $\widetilde{\nabla}$ with respect to the {\it Hermitian} structure and hence $\widetilde{\nabla}$ belongs to $\widetilde C(\mathcal{E})$. Conversely, for given 
$\widetilde{\nabla} \in \widetilde{C}(\mathcal{E})$ recall from equation (\ref{connection-2}) that $\widetilde{\nabla}(\xi) = \displaystyle\sum_{j=1}^n \widetilde{\nabla}_j(\xi)\otimes \sigma_j\,$ where 
$\widetilde{\nabla}_j : \mathcal{E} \longrightarrow \mathcal{E}$. We define $\Phi^{-1}(\widetilde{\nabla}) = \nabla$ where,
\begin{eqnarray*}
\nabla(\xi) = \displaystyle\sum_{j=1}^n i\widetilde{\nabla}_j(\xi)\otimes e_j
\end{eqnarray*}
An absolutely similar computation shows the compatibility of $\nabla$. So elements of $C(\mathcal{E})$ and $\widetilde{C}(\mathcal{E})$ are in one-one correspondence. Recall from equation (\,\ref{YM}\,), 
for finitely generated projective $\mathcal{A}_\varTheta$-module $\mathcal{E}=p\mathcal{A}_{\varTheta}^q\,$ we obtained for $\nabla \in C(\mathcal{E})$,
\begin{eqnarray*}
\textit{YM}(\nabla) = \displaystyle\sum_{j<k}\, \widetilde{\tau}([\nabla_j,\nabla_k]^*[\nabla_j,\nabla_k])
\end{eqnarray*}
where $\widetilde{\tau}$ was the trace on $End(\mathcal{E})$. For $\Phi(\nabla)=\widetilde{\nabla}$ we obtain from Proposition (\,\ref{main}\,)\, , 
\begin{eqnarray*}
\textit{YM}(\widetilde{\nabla}) = Tr_\omega(|D|^{-n})\displaystyle\sum_{j<k}\, \tau_q([\nabla_j,\nabla_k]^*[\nabla_j,\nabla_k])
\end{eqnarray*}
where $\tau_q$ was the extended trace of $\tau$ on $M_q(\mathcal{A}_{\varTheta})$. Identifying $End(\mathcal{E})$ with $pM_q(\mathcal{A}_{\varTheta})p \subseteq M_q(\mathcal{A}_{\varTheta})$ we see that both 
$\widetilde{\tau}$ and $\tau_q$ are equal with $\tau \otimes Trace$. Hence follows the equality of Yang-Mills for both the approaches except for the positive 
scalar factor $Tr_\omega(|D|^{-n}) = 2N\pi^{n/2}/(n(2\pi)^n\Gamma(n/2))$.
\end{proof}

\end{document}